\documentclass[a4paper,oneside,portrait,11pt]{AmsArt}

\usepackage[margin=1in]{geometry}

\usepackage[ps,all,arc,rotate]{xy}

\usepackage {graphicx}
\usepackage{amsfonts}
\usepackage{amsthm}
\usepackage{amsmath}
\usepackage{amsfonts}
\usepackage{latexsym}
\usepackage{amssymb}
\usepackage{epsfig,color}
\usepackage{graphicx, amssymb}

\usepackage{hyperref}
\hypersetup{
  bookmarks=true,
}

\newtheorem{definition}[]{Definition} 
\newtheorem{theorem}[definition]{Theorem}

\newtheorem{proposition}[definition]{Proposition}
\newtheorem{remark}[definition]{Remark}

\theoremstyle{definition}

\def\CA{{\mathcal A}}

\def\CJ{{\mathcal J}}

\def\CM{{\mathcal M}}

\newcommand{\C}{\mathbb{C}}

\title{Spectral data for $U(m,m)$-Higgs bundles}
\author{Laura P. Schaposnik }

\address{Mathematical Institute, University of Oxford, 24-29 St. GilesÕ, Oxford, OX1 3LB, UK.}
 
\address{\textit{Current address:}\newline
Mathematisches Institut, Ruprecht-Karls-UniversitŠt Heidelberg , 69120 Heidelberg, Germany}
\email{schaposnik@mathi.uni-heidelberg.de}

\date{\today}

\begin{document}

\baselineskip=1.1\baselineskip

\begin{abstract}
We  define and study spectral data associated to $U(m,m)$-Higgs bundles through the $GL(2m,\C)$ Hitchin fibration.  We give a new interpretation  of the topological invariants involved, as well as a geometric description of  the moduli space.
\end{abstract}
\maketitle

Examples of  $G$-Higgs bundles on a Riemann surface $\Sigma$ of genus $g\geq 2$,  for $G$ a real form of a complex Lie group $G_{c}$, were initially considered in \cite{N1,N5}. 
Whilst the particular case of  $G$-Higgs bundles for the unitary group with signature has received much attention in the past decade (e.g., see \cite{brad} and references therein), the work found in the literature is mostly on the counting of connected components via  Morse theory, where the critical points lie in the most singular fibre of the corresponding Hitchin fibration. By contrast, this paper is dedicated to the study of the geometry of the moduli space of $U(m,m)$-Higgs bundles via the most generic fibres of the $GL(2m,\C)$ Hitchin fibration.

\begin{definition} \label{higgsupp} A $U(m,m)$-Higgs bundle over $\Sigma$ is a pair $(V,\Phi)$, where
 $V=W_1\oplus W_2$ for $W_i$ rank $m$ vector bundles over $\Sigma$, and the Higgs field $\Phi$ is given by
 \begin{eqnarray}\label{Higgs:eq}
        \Phi=\left( \begin{array}
          {cc} 0&\beta\\
\gamma&0
         \end{array}\right), 
        \end{eqnarray}
where $\beta:W_2\rightarrow W_1\otimes K$ and $\gamma:W_1\rightarrow W_2 \otimes K$ are holomorphic sections for $K$ the canonical bundle of $\Sigma$.
\end{definition}

 We define the spectral data for $U(m,m)$-Higgs bundles in Section \ref{data1}, and
 in Section \ref{invariant}  we describe the topological invariants  (Proposition \ref{teo2}) and the Milnor-Wood type inequalities they satisfy, leading to  our main result:
\vspace{0.1 in}
 
\noindent 
\textbf{Theorem \ref{teoprin}.} \textit{There is a 1-to-1 correspondence between isomorphism classes of  $U(m,m)$-Higgs bundles $(W_1\oplus W_2, \Phi)$    for which ${\rm char}(\Phi)$ defines a smooth curve and $\deg W_1 > \deg W_2 $, and triples $(\bar S,E_+,D)$ where }
\begin{itemize}
 \item  \textit{ $\bar \pi:\bar S\rightarrow \Sigma$ is a smooth $m$-fold cover of $\Sigma$ in the total space of $K^{2}$ defined by }
\begin{eqnarray}\xi^{m}+a_{1}\xi^{m-1}+\ldots+a_{m-1}\xi+a_{m}=0,
\nonumber
 \end{eqnarray}
\textit{for $a_{i}\in H^{0}(\Sigma,K^{2i})$ and $\xi$ the tautological section of the pullback of $K^{2};$}
 \item  \textit{$E_+$ is a line bundle on $\bar S$ whose degree is 
$\deg E_+ = \deg W_1 +(2m^{2}-2m)(g-1) ;$}
\item \textit{$D$ is a positive subdivisor of the divisor of $a_{m}$ of degree $M=\deg W_2-\deg W_1+2m(g-1).$}
\end{itemize}
 \vspace{0.1 in}
 
Interchanging $W_1$ and $W_2$, the above correspondence holds for any $U(m,m)$-Higgs bundle for which $\deg W_1\neq \deg W_2$, and in the case of $W_1=W_2$, we show that the construction fails to be 1-to-1.
 For each choice of invariants  $\deg E_+, M$, in Section \ref{defEE}  we give a geometric description of a distinguished component of the moduli space $\CM_{U(m,m)}$ of $U(m,m)$-Higgs bundles that intersects a generic fibre of the classical Hitchin fibration (Theorem \ref{something4}), studying all the components which are known to exist   (see \cite[p.~116]{brad}). 
 
Finally, we   specialize the above results to study the moduli space of $SU(m,m)$-Higgs bundles as sitting inside the smooth fibres of the $SL(2m,\C)$ Hitchin fibration in Section \ref{ap2}.  \\

\emph{Acknowledgements.}  The results presented in this paper form part of the author's DPhil Thesis at the University of Oxford \cite[Chapters 6,9]{laura}. The author is thankful to Nigel Hitchin  for  his supervision, and to Alan Thompson and Frank Gounelas for helpful discussions. This work was funded by Oxford University Press, New College,  QGM  Aarhus,  and the European Science Foundation through the research grant ITGP Short Visit 4655.

\section{\texorpdfstring{SPECTRAL DATA FOR $U(m,m)$-HIGGS BUNDLES}{Spectral data for U(m,m)-Higgs bundles}}\label{data1}

As shown in  \cite{N2}, there is a natural fibration of the moduli space $\CM$ of classical Higgs bundles of rank $2m$ and fixed degree, the so-called Hitchin fibration, which maps isomorphism classes of Higgs bundles to the coefficients of the characteristic polynomial of the Higgs field
\begin{eqnarray}
h:\CM&\rightarrow& \CA:= \bigoplus_{i=1}^{2m} H^0( \Sigma , K^{i}).\label{hit:map}
\end{eqnarray}
The regular fibre of the Hitchin fibration is isomorphic to the Jacobian of a curve $S$  in the total space of $K$ defined by the characteristic polynomial of $\Phi$. In particular, given a line bundle $E$ on $\pi:S\rightarrow \Sigma$, one can obtain a classical Higgs pair $(V,\Phi)$ by taking the direct image
$V:=\pi_{*}E,$
and letting the Higgs field $\Phi$ be the map obtained through the direct image of   the tautological  section $ E\xrightarrow{\eta} E\otimes \pi^* K$.

\begin{remark} \label{rem:fix} From Definition \ref{higgsupp}, it follows that $U(m,m)$-Higgs bundles can be seen as fixed points in $\mathcal{M}$ of the involution
$\Theta: (V,\Phi)\mapsto (V,-\Phi) $
 on  $\mathcal{M}$ corresponding to vector bundles $V$ which have an automorphism sending $\Phi$ to $-\Phi$ conjugate to  
 \[I_{m,m}=\left(\begin{array}{cc}-I_{m}&0\\0&I_{m}\end{array}\right).\]
\end{remark}

The characteristic polynomial of a $U(m,m)$-Higgs field $\Phi$, which from the above remark is invariant under $\Phi\mapsto -\Phi$,  defines the spectral curve $\pi:S\rightarrow \Sigma$ in the total space of $K$ with equation  
\begin{eqnarray}\eta^{2m}+a_{1}\eta^{2m-2}+\ldots+a_{m-1}\eta^{2}+a_{m}=0,\label{curves}\label{eq:spec2}\end{eqnarray}
for $a_{i}\in H^{0}(\Sigma,K^{2i})$. 
%
Throughout  the paper, we shall assume the curve $S$ to be smooth.

 The $2m$-fold cover $S$ has a natural involution $\sigma: \eta \mapsto -\eta$, and so we may  consider the quotient curve $\bar \pi:\overline{S}=S/\sigma \rightarrow \Sigma$ in the total space of $K^{2}$. This is an $m$-fold cover of $\Sigma$ which is smooth since $S$ is assumed to be smooth,  and has equation
\begin{eqnarray}
\xi^{m}+a_{1}\xi^{m-1}+\ldots+a_{m-1}\xi+a_{m}=0,
 \label{curvess}
\end{eqnarray}
for $\xi=\eta^{2}$ the tautological section of $\overline{\pi}^{*}K^{2}$. 
Let  $g_{S}$ and $g_{\overline{S}}$ be the genus of  $S$ and $\overline{S}$,  respectively, and $K_{S}$ and $K_{\bar S}$ their canonical bundles. 
By 
 the adjunction formula, we have  
  $K_{S}\cong \pi^{*}K^{2m}$ and  $K_{\bar S}\cong  \bar \pi^{*} K^{2m} \otimes \bar \pi^*K^{-1}$,   hence
$g_{S}=4m^{2}(g-1)+1,$ and  $g_{\overline{S}}= (2m^{2}-m)(g-1)+1.$

Under certain conditions, the line bundle $E\in {\rm Jac}(S)$ associated to a $GL(2m,\C)$-Higgs bundle defines a $U(m,m)$-Higgs bundle:

\begin{proposition}Let $S$ be a smooth curve in the total space of $K$ as in (\ref{eq:spec2}) with a natural involution $\sigma:\eta\mapsto -\eta$, and $E$ a line bundle on it. Then, $E$ defines a $U(m,m)$-Higgs bundle if and only if $\sigma^* E\cong E$. \label{PropE}
\end{proposition}

\begin{proof} For $S,E$ as above, there are two possible lifts of the action of $\sigma$ to $E$ which differ by $\pm 1$. By abuse of notation we shall denote the choice of lifted action on $E$  also  by $\sigma$. On an invariant open set  $\pi^{-1}(\mathcal{U})\subset S$  we may decompose the sections of $E$ into the invariant and anti-invariant parts, labeled by the upper indices $\pm$ as follows
\[H^{0}(\pi^{-1}(\mathcal{U}),E)=H^{0}(\pi^{-1}(\mathcal{U}),E)^{+} \oplus H^{0}(\pi^{-1}(\mathcal{U}),E)^{-}.\]
From the definition of direct image there is an equivalent  decomposition of $H^{0}(\mathcal{U},\pi_{*}E)$ into
$H^{0}(\mathcal{U},\pi_{*}E)=H^{0}(\mathcal{U},\pi_{*}E)^{+}\oplus H^{0}(\mathcal{U},\pi_{*}E)^{-}.$
Hence 
$\pi_{*}E=W_{+}\oplus W_{-},$
where $W_{\pm}$ are vector bundles on $\Sigma$. At a point $x$ such that $a_m(x)\ne 0$, the involution $\sigma$ has no fixed points on $\pi^{-1}(x)$. Moreover,  if $x$ is not a branch point, $\pi^{-1}(x)$ consists of $2m$ points $e_1,\dots, e_m, \sigma e_1,\dots \sigma e_m$. The fibre of the direct image is then isomorphic to $\mathbb{C}^m\oplus \mathbb{C}^m$, and the involution is $\sigma: (v,w)\mapsto (w,v)$. Then, the fibre of $W_+$ is given by the invariant points $(v,v)$, and the one of $W_{-}$ is given by the anti-invariant points $(v,-v)$, and so ${\rm rk} W_+ ={\rm rk} W_- =m$.

The Higgs field associated to $E$ is defined as in the case of classical Higgs bundles through  the multiplication map 
$E\xrightarrow{\eta} E\otimes \pi^{*}K.$
Since  $\sigma(\eta)=-\eta$, the Higgs field $ \Phi$ maps $W_{+}\mapsto W_{-}\otimes K$ and $W_{-}\mapsto W_{+}\otimes K$ and  thus it may be written  as in (\ref{Higgs:eq}).

Conversely,  as a classical Higgs bundle,  a $U(m,m)$-Higgs bundle $(V,\Phi)$ with smooth spectral curve $S$ as in (\ref{eq:spec2}) has associated a line bundle $E$ on $S$ defined as the cokernel of $\eta I -\pi^*\Phi$ in $\pi^*V\otimes \pi^*K$.  The involution $\sigma$ transforms the line bundle $E$ for eigenvalue $\eta$ to $\sigma^* E$ for eigenvalue $-\eta$. Furthermore, as noted before the isomorphism classes of $U(m,m)$-Higgs bundles are fixed by $\Theta:(V,\Phi)\mapsto (V, -\Phi)$    and thus
 $\sigma^{*}E\cong E$. \end{proof}

Following \cite{brad}, we say that a $U(m,m)$-Higgs bundle $(V,\Phi)$ as in Definition \ref{higgsupp} is stable if for every subbundle $V'=W_1'\oplus W_2'$ with $W_1'\subset W_1$ and $W_2'\subset W_2$, satisfying $\Phi(V')\subset V'\otimes K$, one has that
$ \deg V'/{\rm rk}V'< \deg V/{\rm rk}V.$ 
Since in Proposition \ref{PropE} we have considered an irreducible  curve $S$,  there is no proper subbundle of $V$ which is preserved by $\Phi$, and hence the $U(m,m)$-Higgs pair constructed through $E$ is stable.

 By means of the 2-fold cover $\rho:S \rightarrow \bar S$, the line bundle $E$ associated to a $U(m,m)$-Higgs bundle $(W_1\oplus W_2,\Phi)$ can be seen in terms of line bundles on   $\bar S$. The invariant and anti-invariant sections of $E$ give a decomposition
$H^{0}(\rho^{-1}(\mathcal{V}),E)=H^{0}(\rho^{-1}(\mathcal{V}),E)^{+}\oplus H^{0}(\rho^{-1}(\mathcal{V}),E)^{-}$, for $\mathcal{V}\subset \overline{S}$ an open set.
Then, by definition of direct image, there are two line bundles $E_+$ and $E_-$ on the quotient curve $\overline{S}$ such that  
$H^{0}(\rho^{-1}(\mathcal{V}),E)^{\pm}\cong H^{0}(\mathcal{V},E_\pm)$.
Moreover, by similar arguments as above, one has  $\rho_{*}E=E_+\oplus E_-$ on $\bar S$ for which $\bar \pi_* E_+ = W_1$ and $\bar \pi_* E_- =W_2$.  
 
\begin{remark}
Considering the maps $\beta\gamma$ and $\gamma\beta$, one obtains $K^2$-twisted Higgs bundles whose associated spectral curve is $\bar S$ (e.g., see  \cite{bobi}), with corresponding line bundles  $E_\pm$  on $\bar S$.
\end{remark}

\section{\texorpdfstring{THE ASSOCIATED INVARIANTS }{The associated invariant}}\label{invariant}

The topological invariants associated to $U(m,m)$-Higgs bundles as in Definition \ref{higgsupp} are the degrees $\deg W_1$ and $\deg W_2$. These invariants arise in a natural way in terms of the spectral data from the isomorphism  $\sigma^{*}E\cong E$ of the  line bundle obtained through Proposition \ref{PropE} and the quotient curve $\bar S$. At a fixed point $a\in S$ of the involution, there is a linear action of $\sigma$ on the fibre $E_a$ given by scalar multiplication of $\pm 1$.   We shall denote by $M$ the number of points at which the action of $\sigma$ is $-1$ on $E$.

\begin{proposition}\label{teo2}
The topological invariants of a $U(m,m)$-Higgs bundle $(W_1\oplus W_2, \Phi)$ with corresponding line bundle $E$ can be expressed in terms of the spectral data as follows
\begin{eqnarray}
  \deg W_1 &=&\frac{\deg E-M}{2}+(2m-2m^{2})(g-1),\label{degw1}\\
\deg W_2 &=&\frac{\deg E +M}{2}-2m^{2}(g-1). \label{degw2}
\end{eqnarray}
\end{proposition}

\begin{proof}  From Proposition \ref{PropE}, the involution $\sigma$ preserves the line bundle $E$ and over its $4m(g-1)$ fixed points, which are the zeros of $a_{m}$ in (\ref{eq:spec2}),  it acts as $\pm 1$.   
Furthermore, over the fixed points   $\sigma$ acts as $+1$ on $\rho^{*}L$, for any line bundle $L$   on $\bar S$.  Hence, the involution acts as $-1$ on  $E\otimes \rho^{*}L$ over exactly $M$ points. Choosing $L$ of sufficiently large degree such that $H^{1}(S,E\otimes \rho^{*}L)$ vanishes, by Riemann-Roch and Serre duality we have that
 \begin{eqnarray}\dim H^{0}(S,E\otimes \rho^{*}L)^+ +\dim H^{0}(S,E\otimes \rho^{*}L)^{-}=\deg E+\deg L-4m^{2}(g-1),\label{mvw2}\end{eqnarray}
where as before the $\pm$ upper script labels  the $\pm1$ eigenspaces of $\sigma$ in $H^{0}(S,E\otimes \rho^*L)$.
From  \cite[Theorem 4.12]{at},  the holomorphic Lefschetz theorem gives us 
\begin{eqnarray}
 \dim  H^{0}(S,E\otimes \rho^{*}L)^+ - \dim H^{0}(S,E\otimes \rho^{*}L)^-=\frac{(-M)+(4m(g-1)-M)}{2}.
\end{eqnarray}
Hence, from (\ref{mvw2}) we have 
\begin{eqnarray}
\dim  H^{0}(S,E\otimes \rho^{*}L)^+&=&
 \frac{\deg E+\deg L-M}{2}+(m-2m^{2})(g-1),\nonumber\\
 \dim  H^{0}(S,E\otimes \rho^{*}L)^{-}&=&
\frac{\deg E+\deg L +M}{2}-(m+2m^{2})(g-1).\nonumber
\end{eqnarray}
Finally, recall that  $\rho_*E=E_+\oplus E_-$ satisfying $\bar \pi_* E_+=W_1$ and $\bar \pi_* E_-=W_2$. Since by construction $\dim H^0 (S,E\otimes \rho^{*}L)^\pm=\dim H^0(\bar S,E_\pm \otimes L) $,  applying Riemann-Roch and Serre duality,  we obtain (\ref{degw1}) and (\ref{degw2}) as required, and also  note that $M=\deg W_2-\deg W_1+2m(g-1).$
\end{proof}

\begin{remark}\label{mvar}\label{rem:sigma}
 Throughout this section  we have assumed that $W_1$ and $W_2$ are obtained via the invariant and anti-invariant sections of $E$, respectively. Interchanging the role of $W_1$ and $W_2$ corresponds to considering the involution $-\sigma$ acting on $S$, and the associated invariants $\deg E$ and $\overline{M}=4m(g-1)-M.$
\end{remark}

 For $m>1$,   the degrees of the bundles $E_+$ and $E_-$  are given by\begin{eqnarray}
 \deg E_+ &=  \deg W_1 +(2m^{2}-2m)(g-1) &= \frac{\deg E}{2}-\frac{M}{2},\label{degu1}\\
 \deg E_- &=  \deg W_2 +(2m^{2}-2m)(g-1) &= \frac{\deg E}{2}+\frac{M}{2}-2m(g-1).\label{degu2}
\end{eqnarray}
\begin{remark}The parity of the   $\deg E$ and   $M$  need to be the same.  \end{remark}

Given a stable $U(m,m)$-Higgs field $\Phi=(\beta,\gamma)$  as in Definition \ref{higgsupp},  if $\deg W_1\geq \deg W_2$ (else one can interchange $W_1$ and $W_2$),  then $\gamma\neq 0$ because otherwise the invariant subbundle $W_1$   would contradict stability.
Moreover, the section $\gamma$ induces a non-zero map $c\in H^{0}(\overline{S},E_+^*\otimes E_-\otimes \overline{\pi}^*K)$ obtained through the multiplication map $H^{0}(S,E)^+\xrightarrow{\eta} H^{0}(S,E\otimes \pi^* K)^{-}$. Hence, since $c$ is odd, it must vanish at the fixed points over which $\sigma$ acts as $-1$ on $E$. 

Since $\bar \pi$ is a degree $m$ cover, from (\ref{degu1}) and (\ref{degu2}) we have that $\deg (E_+^*\otimes E_-\otimes \overline{\pi}^*K)=M $. Therefore the section $c$ vanishes only at the $M$ points over which $\sigma$ acts as $-1$ on $E$, and defines a positive subdivisor $D$ of the divisor of  $a_{m}={\rm det}( \beta \gamma)$ giving the whole set of fixed points.
For $[D]$ the line bundle on $\bar S$ associated to the divisor $D$, it follows that 
\begin{eqnarray}[D]=E_+^*\otimes E_-\otimes \overline{\pi}^*K. \label{divisorD}\end{eqnarray}
 \begin{remark}For  $m=1$ the surface $\Sigma$ and the curve $\overline{S}$ coincide, and the spectral data is studied in \cite{N1}.  
\end{remark}


From the  study of the spectral data for $U(m,m)$-Higgs bundles,  we have the following:

\begin{theorem}\label{teoprin}
There is a 1-to-1 correspondence between isomorphism classes of  $U(m,m)$-Higgs bundles $(W_1\oplus W_2, \Phi)$ on $\Sigma$ with non-singular spectral curve  for which $\deg W_1 > \deg W_2 $, and triples $(\bar S,E_+,D)$ up to isomorphism,  where 
\begin{itemize}
 \item   $\bar \pi:\bar S\rightarrow \Sigma$ is a smooth $m$-fold cover of $\Sigma$ in the total space of $K^{2}$ defined by 
\begin{eqnarray}\xi^{m}+a_{1}\xi^{m-1}+\ldots+a_{m-1}\xi+a_{m}=0,
\label{eq:barS}\end{eqnarray}for $a_{i}\in H^{0}(\Sigma,K^{2i})$ and $\xi$ the tautological section of the pullback of $K^{2};$
 \item  $E_+$ is a line bundle on $\bar S$ whose degree is 
$\deg E_+ = \deg W_1 +(2m^{2}-2m)(g-1) ;$
\item \ $D$ is a positive subdivisor of the divisor of $a_{p}$ of degree $M=\deg W_2-\deg W_1+2m(g-1).$
\end{itemize}

 \end{theorem}
\begin{proof}
 Starting with a $U(m,m)$-Higgs bundle, consider the line bundle $E$ as in Proposition \ref{PropE}, and let  $D$ be the positive subdivisor of the divisor of $a_{m}$ over which $\sigma$ acts as $-1$ on $E$ as in Proposition \ref{teo2}. As seen previously, the direct image  $\rho_{*}E$ on $\bar S=S/\sigma$ decomposes into  $\rho_{*}E=E_+\oplus E_-$, for $E_\pm$ line bundles on $\bar S$ satisfying 
$[D]=E_+^{*}\otimes E_-\otimes \bar \rho^{*}K. $
Note that  from equations (\ref{degu1})-(\ref{degu2}) since $\deg W_1>\deg W_2$ one has that $\deg E_+>\deg E_-$. Hence, considering the line bundle $E_+$ of biggest degree, and identifying ${\rm Pic}^{\deg E_+}(\bar S)$ with ${\rm Jac}(\bar S)$, one can construct the triple $(\bar S,E_+,D)$. 


Conversely, given a triple $(\bar S,E_+,D)$, since the curve $\bar S$ is smooth, the section $a_m$ has simple zeros  \cite{bobi}, and   we may  write
$[a_{m}]=D+\overline{D},$ 
for $\overline{D}$ a positive divisor on $\Sigma$. Furthermore, following (\ref{divisorD}) we define the line bundle $E_-$ on $\bar S$ by
$E_-= [D]\otimes E_+ \otimes \bar\pi^{*}K^{*}.$
One should note that in the case of $\deg W_1=\deg W_2$ the number of points in $D$ and $\bar D$ is the same, and thus the line bundle $E_-$ could be constructed through either of the two divisors, making the correspondence fail to be 1-to-1 in this case. 

On the curve $\bar S$ we may consider the sections associated to the divisors $D$ and $\bar D$, which  induce the natural maps
$
\bar \beta: E_-  \xrightarrow{} [\bar D]\otimes E_- = E_+  \otimes \bar\pi^{*}K $ and
$\bar \gamma: E_+ \xrightarrow{}  [D]\otimes E_+ =  E_-  \otimes \bar\pi^{*}K $
Then,  there is a natural rank $2$ Higgs bundle $(E_+\oplus E_-,\bar \Phi)$ whose Higgs field $\bar \Phi$  has  off diagonal entries $\bar \beta$  and $\bar \gamma$. Moreover, $\bar \beta \bar \gamma$ is given by $a_m$ up to scalar multiplication, and  the spectral curve of this $\bar \pi^*K$-twisted Higgs bundle $\bar \Phi$ is a curve $S$ whose quotient under $\sigma:\eta \mapsto -\eta$ gives $\bar S$.

Following the methods for classical Higgs bundles, $(E_+\oplus E_-,\bar \Phi)$ has an associated line bundle $E$ on $S$ which is preserved by the involution $\sigma$ on $S$ (e.g., \cite{bobi} ) and such that   $\rho_{*}E=E_+\oplus E_-$ via the invariant and anti-invariant sections and by similar arguments as the ones leading to (\ref{divisorD}), the involution $\sigma$ acts as $-1$ on $E$ over the divisor $D$, whence proving the proposition. 
\end{proof}   

\begin{remark}
By interchanging the involutions $\sigma$ and $-\sigma$, one can show that an equivalent correspondence exists in the case of $U(m,m)$-Higgs bundles for which $\deg W_1 < \deg W_2$.
\end{remark}

Following \cite{brad}, the Toledo invariant $\tau(\deg W_1,\deg W_2)$ associated to $U(m,m)$-Higgs bundles is  defined as $\tau (\deg W_1,\deg W_2):=\deg W_1-\deg W_2, $ and so from the previous calculations, this invariant  may be expressed as
$\tau (\deg W_1,\deg W_2)=-M+2m(g-1).$
 By definition,  $M$  satisfies
$0\leq M\leq 4m(g-1)$, and thus
\begin{eqnarray}0\leq |\tau (\deg W_1,\deg W_2)|\leq 2m(g-1),\end{eqnarray}
which agrees  with the  bounds given in \cite{brad} for the Toledo invariant.  

\begin{remark}
 The maximal Toledo invariant corresponds to $M=0$ and thus  from Section \ref{invariant} in this case  the non-zero map $c\in H^{0}(\overline{S},E_+^*\otimes E_-\otimes \overline{\pi}^*K)$ associated to a $U(m,m)$-Higgs bundle is an isomorphism. Then, the section $\gamma$ is nowhere vanishing, giving an isomorphism $W_1\cong W_2 \otimes K$. For $L_0$ a choice of square root of $K$ one can construct a Cayley pair $(\tilde{W_2},\tilde \theta)$,  a $K^2$-twisted Higgs bundle on $\Sigma$ naturally associated to the $U(m,m)$-Higgs bundle \cite{brad}.
 This is done by considering
   $\tilde \theta= [(\gamma \otimes I_K) \circ \beta]\otimes I_{L_0}~ :~ W_2\otimes L_0\rightarrow W_2\otimes L_0\otimes K^2$, for $I_K$ and $I_{L_0}$ the identities on $K$ and $L_0$. Moreover, the spectral curve of the Cayley pair is $\bar S$, the corresponding line bundle  is $E_-\otimes \bar \pi^* L_0 $, and  Theorem \ref{teoprin} provides a realisation of the Cayley correspondence. 
 
   \end{remark}

\begin{remark}
The methods developed here to identify the topological invariants associated to $U(p,p)$-Higgs bundles in terms of an action on fixed points can be extended to the study of Higgs bundles for other real forms, and we do this in \cite{laura} and \cite{nonabelian}. 
\end{remark}
\newpage
%


\section{\texorpdfstring{THE MODULI SPACE $\mathcal{M}_{U(m,m)}$}{The moduli space of U(m,m)-Higgs bundles}}\label{ap1}
\label{defEE}


Through dimensional arguments one can show that the space of isomorphism classes of $U(m,m)$-Higgs bundles satisfying Theorem  \ref{teoprin} is included in the moduli of stable $U(m,m)$-Higgs bundles $\mathcal{M}_{U(m,m)}$ as a Zariski open set. We shall denote by $\CA' \subset \CA$ the points in the Hitchin base of the classical Hitchin fibration over which $\CM_{U(m,m)}$ lies, this is, 
 \begin{eqnarray}\CA'=\bigoplus_{i=1}^{m}H^{0}(\Sigma,K^{2i}).\label{aprime}\end{eqnarray}

 \begin{proposition}
The dimension of the parameter space of the triples $(\bar S,E_+, D)$ associated
to $U(m,m)$-Higgs bundles as in Theorem \ref{teoprin} is, as expected,  $4m^{2}(g-1)+1$.
\end{proposition}
\begin{proof}
 Since by Bertini's theorem a generic point in $\CA'$ gives a smooth curve $\bar S$, the  parameter space of $\bar S$  has  dimension $\dim \CA'=(2m^{2}+m)(g-1)$.
The choice of the divisor $D$ gives a partition of the zeros of $a_{m}$, and from Theorem \ref{teoprin}, the choice of $E_+$ is given by an element in ${\rm Jac}(\bar S)$, which has dimension
$g_{\bar S}= 1+(2m^{2}-m)(g-1).$
Thus, the parameter space  of $(\bar S,E_+,D)$ has dimension 
$(2m^{2}+m)(g-1)+1+(2m^{2}-m)(g-1)=4m^{2}(g-1)+1.$
Moreover, 
${\rm dim} \mathcal{M}_{GL(2m,\mathbb{C})} ={\rm dim}GL(2m,\mathbb{C})(g-1)=8m^{2}(g-1)+2,$
 and the expected dimension of the moduli space of $U(m,m)$-Higgs bundles is
$4m^{2}(g-1)+1.$ 

Finally, since the the discriminant locus in the base $\CA'$ which guarantees a smooth spectral curve  is given by algebraic equations, the   space of isomorphism classes of $U(m,m)$-Higgs bundles satisfying  Theorem \ref{teoprin} is included in the moduli space $\mathcal{M}_{U(m,m)}$ as a Zariski open set.
 \end{proof}

From the construction of the spectral data in Theorem \ref{teoprin} we have the following:

\begin{theorem}\label{something4}
Each pair of invariants $\deg E$, $M$ as in Proposition \ref{teo2} labels exactly one connected component of the moduli space of $U(m,m)$-Higgs bundles which intersects the non-singular fibres of the Hitchin fibration
$\mathcal{M}_{GL(2m,\mathbb{C})}\rightarrow \mathcal{A}_{GL(2m,\mathbb{C})}.$
This component is given by a fibration of Jacobians over the total space of a vector bundle on the symmetric product $S^M\Sigma$. \end{theorem}

\begin{proof}
In order to describe the connected component of $\CM_{U(m,m)}$ which intersects the regular fibres of the classical Hitchin fibration for fixed invariants $\deg E$ and $M$, we shall analyse the parameter space for the choices of $(\bar S, D, E_+)$. 
From its construction, the choice of the curve $\bar S$ is given by the space
$\CA'$ as in (\ref{aprime}). 
Then, since the degree $\deg E_+$ is fixed,  the choice of $E_+$ is given by a fibration $\CJ$ whose fibres are the Jacobians ${\rm Jac }(S_a)$ of the smooth curves $S_a$ defined by each $a\in \CA'$. 
We shall see now that the choice of the divisor $D$ of degree $M$ is equivalent to replacing $H^{0}(\Sigma, K^{2m})$ in (\ref{aprime}) by a vector bundle over the symmetric product $S^M \Sigma$.

Let  $a_m$ be a fixed differential corresponding to some  $a=(1_1,\dots, a_m)\in\CA'$ for $a_i\in H^{0}(\Sigma, K^{2i})$. In order to construct the divisor $D$, one has to select  a subdivisor of $a_{m}$ consisting of $M$ points, such that $[a_m]= D+\bar D$ for a positive divisor $\bar D$, and hence choose  $x\in S^M \Sigma$. In particular, since it is a ramification subdivisor, $D$ and $\bar D$ can be also thought of as divisors on $S$ and $\bar S$.

Following the proof of Theorem \ref{teoprin} and the ideas in \cite[Proposition 10.2]{N1}, the divisor $D$ defines $\bar \gamma: E_+ \xrightarrow{}  [D]\otimes E_+$ up to scalar multiple. The fibre of this point $x\in S^M \Sigma$ is the choice of $\bar \beta $ up to scalar multiple, which corresponds to $\bar D$,  and  gives us a point in a vector bundle $B$ of rank $(4m-1)(g-1)-M$ over $S^M\Sigma$ whose fibres are $H^0(\Sigma, K^{2m}\mathcal{O}(-x))$, for $x\in S^M \Sigma$.

Starting with an effective divisor $D$ given by $M$ points in $\Sigma$,  consider the  exact sequence 
$0\rightarrow K^{2m} (-D) \rightarrow K^{2m}  \rightarrow K^{2m}|_D  \rightarrow 0,$
and the corresponding sequence of cohomology groups
\[0\rightarrow H^0(\Sigma, K^{2m}(-D))\rightarrow H^{0}(\Sigma, K^{2m})\rightarrow H^{0}(D, K^{2m})\rightarrow H^{1}(\Sigma, K^{2m}(-D))\rightarrow \ldots\]
Note that since $\deg D = M < 2m(g-1)$ and $m>1$, by Riemann Roch and Serre duality    $\dim H^1(\Sigma, K^{2m}(-D)) =0$, leading to the short exact sequence
 \[0\rightarrow H^0(\Sigma, K^{2m}(-D))\xrightarrow{f_1} H^{0}(\Sigma, K^{2m})\xrightarrow{f_2} H^{0}(D, K^{2m})\rightarrow 0.\]
 Then, for a section  $s\in H^0(\Sigma, K^{2m}(-D))$ in the $(4m-1)(g-1)-M$-dimensional fibre over $D$, the injective map $f_1$ defines a corresponding differential $a_m:=f_1(s)\in H^0(\Sigma, K^{2m})$.
 
 From the above analysis, the connected parameter space of the triples $(\bar S, D, E_+)$ for fixed topological invariants can be obtained by replacing $H^0(\Sigma, K^{2m})$ in the base (\ref{aprime}) by the vector bundle $B$ over $S^M \Sigma$, and looking at the Zariski open set there which gives smooth curves $\bar S$.    \end{proof}
%

   \section{\texorpdfstring{ SPECTRAL DATA FOR $SU(m,m)$}{Spectral data for $SU(m,m)$}}\label{ap2}

An $SU(m,m)$-Higgs bundle is given by a $U(m,m)$-Higgs bundle $(W_1\oplus W_2,\Phi)$ for which $\Lambda^{m}W_1\cong \Lambda^{m}W_2^{*}$.
In particular, we have that $\deg W_1=-\deg W_2$ and thus one can adapt Theorem \ref{teoprin}  to obtain   $\deg E=(4m^{2}-2m)(g-1),$ and $M = 2 \deg W_2 +2m(g-1)$.
Moreover, 
$\Lambda^{m}\overline{\pi}_{*}E_+\cong
\Lambda^{m}\overline{\pi}_{*}E_-.$ 
Considering the
norm map $Nm:{\rm Pic}(\overline{S})\rightarrow {\rm Pic}(\Sigma)$,  from \cite[Section 4]{bobi} the determinant bundles are given by $ \Lambda^{m}\overline{\pi}_{*}E_\pm= {\rm Nm}(E_\pm)\otimes K^{-m(m-1)} 
$, identifying divisors of $\Sigma$ and their
corresponding line bundles, and thus
${\rm Nm}(E_+)= -{\rm Nm}(E_-)+2m(m-1)K.$ 
 In terms of divisors on $\Sigma$,  one has
$
 D={\rm Nm}~E_+^{*}+{\rm Nm}~ E_- + mK.\label{uno}
$
Therefore, the condition for the spectral data of Theorem \ref{teoprin} to give an $SU(m,m)$-Higgs bundle can be expressed as
\begin{eqnarray}
2{\rm Nm}~E_+=m(2m-1)K-D,\label{ecu1}
\end{eqnarray}
or equivalently $ {\rm Nm}~([D]\otimes E^{2}_{+}\otimes\overline{\rho}^{*}K^{1-2m})=0. $ 
The choice of $E_+$ is thus determined by the choice of an element in $\ker ( {\rm Nm})$, i.e., in ${\rm Prym}(\overline{S},\Sigma)$.
 As in the case of  $U(m,m)$-Higgs bundles, in this case the choice of $D$  is given by a point in a symmetric product of  $\Sigma$,
and the section $a_{m}$ is given by the divisor $D$ together with a section of $K^{2m}[-D]$ (whose divisor is $\bar D$).

 Since  for maximal Toledo invariant  $\tau(\deg W_1,\deg W_2)=\deg W_2-\deg W_1= 2m(g-1),$   in this case the divisor $D$ 
 is empty and thus (\ref{ecu1}) reduces to
$2{\rm Nm}E_+= m(2m-1)K.$ 
Hence, for each choice of square root of $m(2m-1)K$, 
the line bundle $E_+$  is determined by an element in the Prym variety ${\rm Prym}(\bar S, \Sigma)$. From Section \ref{defEE}, in this case the choice of $\bar S$ and $D$ is given by a point in a Zariski open subset of $\CA'$ as in (\ref{aprime}).  Therefore, for each of the $2^{2g}$ choices one has a copy of  ${\rm Prym}(\bar S,\Sigma)$, giving a fibration over the Zariski open in $\CA'$ whose fibres are the disjoint union $2^{2g}$ Prym varieties. Note that if $m$ is even, there is a distinguished choice of square root. 
%
%
%
%
For $m=1$, the study of these components is done through the monodromy action in \cite{lau11}. 

 \begin{remark}
 As noted in \cite[Remark 31]{ABA}, the spectral data defined in previous sections  can be used,  in physics terminology,  to give a geometric realisation of $(B,A,A)$-branes and understand their intersection with the $(A,B,A)$-branes defined in the above paper. 
 \end{remark}

\end{document}